\documentclass{amsart}
\usepackage{amsmath, amssymb,amsthm,latexsym}
\usepackage{enumerate}

\newtheorem{theorem}{Theorem}[section]
\newtheorem{lemma}[theorem]{Lemma}

\theoremstyle{definition}

\newtheorem{remark}[theorem]{Remark}

\title{Presentations for subsemigroups of $PD_n$}

\author{Abdullahi Umar}

\address{Department of Mathematics, College of Arts and Sciences, Petroleum Institute, P. O. Box 2533,
Abu Dhabi, U. A. E.}

\email{\texttt{\{aumar\}@pi.ac.ae}}

\begin{document}

\begin{abstract}
Let $[n]=\{1,\ldots,n\}$ be the $n$-chain.
We give presentations for the following transformation semigroups: the semigroup of full order-decreasing mappings of $[n]$, the semigroup of partial one-to-one order-decreasing mappings of $[n]$, the semigroup of full order-preserving and order-decreasing mappings of $[n]$, the semigroup of partial one-to-one order-preserving and order-decreasing mappings of $[n]$, and the semigroup of partial order-preserving and order-decreasing mappings of $[n]$.
\end{abstract}



\maketitle

\section{Introduction and Preliminaries}

This was supposed to be joint article with Dr. Victor Maltcev, but he does not want to be named as a coauthor. There are many well-known presentations for various finite permutation groups, for example, the famous Coxeter presentation for the symmetric group. After tools of Semigroup Theory were developed, semigroup theorists became interested in finding presentations for finite transformation semigroups. Some of the early works on this go back to the start of earnest investigation as to what semigroup presentations are, see~\cite{Nik,Nik-thesis}.
Later some papers appeared on presentations for order-preserving mappings, see~\cite{Vitor,Vitor2}. Inspired by the Coxeter groups and Brauer-type semigroups, there also appeared presentations for various subsemigroups of the composition monoid in ~\cite{Ania} and~\cite{Mazor}. A substantial amount of work to find presentations for various transformation-like monoids has been done in the papers~\cite{trio}--\cite{East1}.

In general, the big motivation for finding presentations for semigroups lies in the fact that knowing their presentations is really helpful to construct various type of representations of those semigroups, see~\cite{SashaVolodia} and~\cite{Stei}. But for us the main objective in this paper is purely combinatorial -- we are going to provide presentations for various subsemigroups of $PD_n$, which we now define.

Consider an $n$-chain, i.e. $[n]=\{1,\ldots,n\}$ together with the the natural linear order on it. Let $P_n$ denote the semigroup of all partial mapping on $[n]$. By an \emph{order-decreasing} mapping we shall mean a partial mapping $\alpha\in P_n$ such that $x\alpha\leq x$ for all $x\in\mathrm{dom}(\alpha)$; and by an \emph{order-preserving} mapping we shall mean a partial mapping $\alpha\in P_n$ such that $x\alpha\leq y\alpha$ for all $x\leq y$ from the domain of $\alpha$. Let
\begin{itemize}
\item
$PD_n$ be the semigroup of all partial order-decreasing mappings of the chain $[n]$;
\item
$D_n$ be the semigroup of all full order-decreasing mappings of $[n]$;
\item
$PC_n$ be the semigroup of all partial order-preserving and order-decreasing mappings of $[n]$;
\item
$C_n$ be the semigroup of all full order-preserving and order-decreasing mappings of $[n]$;
\item
$IC_n$ be the semigroup of all partial one-to-one order-preserving and order-decreasing mappings of $[n]$;
\item
$ID_n$ be the semigroup of all partial one-to-one order-decreasing mappings of $[n]$.
\end{itemize}
The subsemigroups of $PD_n$ and $D_n$ are well-studied, one should consult~\cite{Higgins},
\cite{Abdullahi}, \cite{Abdullahi2}, \cite{Abdullahi-thesis}.

Now section by section we will provide the presentations for these semigroups. But prior to that, we recall what semigroup presentations are: Let $A$ be a finite set, and $A^{\ast}$ be the free monoid generated by $A$, i.e. the set of all words over $A$ under concatenation. Let $R\subseteq A^{\ast}\times A^{\ast}$.
Then by the (monoid) \emph{presentation} $\langle A:R\rangle$ we mean the semigroup $A^{\ast}/\rho$, where $\rho$ is the congruence on $A^{\ast}$, defined as follows: for $u,v\in A^{\ast}$, we have $u\rho v$ if and only if there exist $x_0,x_1\ldots,x_n\in A^{\ast}$ such that $u=x_0\sim x_1\sim\cdots\sim x_{n-1}\sim x_n=v$, where by $x\sim y$ for $x,y\in A^{\ast}$ we mean that there exists $(u,v)\in R$ and $\alpha,\beta\in A^{\ast}$ such that either $x=\alpha u\beta$ and $y=\alpha v\beta$, or $x=\alpha v\beta$ and $y=\alpha u\beta$. For more detail on monoid presentations we refer the reader to the two theses~\cite{Nik-thesis} and~\cite{Me}.

\section{Presentation for $D_n$}

A general study of the semigroups $D_n$ and $PD_n$ was initiated in \cite{Abdullahi-thesis} and they
arise in language theory \cite{Higgins}.

Prior to stating the presentation for $D_n$, let us discuss its natural generating set. For $1\leq i<j\leq n$ let $f_{i,j}$ be the idempotent in $D_n$ which maps $j$ to $i$ and fixes all remaining points. Then one checks that the $f_{i,j}$-s generate $D_n$ (see~\cite{Abdullahi}) and are subject to the following relations (under the mapping $e_{i,j}\mapsto f_{i,j}$):
\begin{eqnarray}
e_{i,j}^2 &=& e_{i,j}\label{eq:1}\\
e_{i,j}e_{k,l}&=&e_{k,l}e_{i,j}\quad\mathrm{if}~\{i,j\}\cap\{k,l\}=\varnothing\label{eq:2}\\
e_{i,j}e_{i,k}&=&e_{j,k}e_{i,j}\quad i<j<k\label{eq:3}\\
e_{i,k}e_{i,j}&=&e_{j,k}e_{i,j}\quad i<j<k\label{eq:4}\\
e_{i,k}e_{j,k}&=&e_{i,k}~\quad\quad i,j<k.\label{eq:5}
\end{eqnarray}
We will prove now that the monoid $M_n$, presented by~\eqref{eq:1}--\eqref{eq:5}, is isomorphic to $D_n$.
Firstly, we start with
\begin{lemma}\label{lm:simple-but-great}
Let $i<j$ and $k<l$. If $l<j$, then we can select $p,q,r$ such that $p<q$, $r<j$, $\{p,q,r\}=\{i,k,l\}$ and $e_{i,j}e_{k,l}=e_{p,q}e_{r,j}$.
\end{lemma}

\begin{proof}
Let $l<j$.
If $\{k,l\}\cap\{i,j\}=\varnothing$, then the claim follows from~\eqref{eq:2}. If $\{k,l\}\cap\{i,j\}\neq\varnothing$, then one of $k$ and $l$ must coincide with $i$. If $l=i$, then $k<l=i<j$ and so $e_{i,j}e_{k,l}=e_{i,j}e_{k,i}=e_{k,i}e_{k,j}$ by~\eqref{eq:3}. If $k=i$, then $i=k<l<j$ and so $e_{i,j}e_{k,l}=e_{i,j}e_{i,l}=e_{i,l}e_{i,j}$ by~\eqref{eq:4} and~\eqref{eq:3}. The claim follows.
\end{proof}

Now we will show that every element from $M_n$ can be represented in the form
\begin{equation}\label{eq:the-form}
e_{i_1,j_1}e_{i_2,j_2}\cdots e_{i_k,j_k}
\end{equation}
for some $k\geq 0$, $2\leq j_1<j_2<\cdots<j_k\leq n$ and $1\leq i_s<j_s$ for all $s\leq k$.

By inductive arguments, to establish this it suffices to show that the product of the element~\eqref{eq:the-form} by any element $e_{i,j}$ is of the form~\eqref{eq:the-form}. So, let $\pi=e_{i_1,j_1}e_{i_2,j_2}\cdots e_{i_k,j_k}$ with all the above conditions on $i_s$ and $j_s$. If $j>j_k$, then $\pi e_{i,j}$ is of the required form. If $j=j_k$, then, using~\eqref{eq:5}, we have that $\pi e_{i,j}=\pi$ is again of the required form. So, let $j<j_k$. Then by Lemma~\ref{lm:simple-but-great} there exist $p,q,r<j_k$ such that $e_{i_k,j_k}e_{i,j}=e_{p,q}e_{r,j_k}$. Repeating this at most $k-1$ more times will bring $\pi e_{i,j}$ to the required form.

One now notices that there are exactly $n!$ different formal products~\eqref{eq:the-form}, which yields $|M_n|\leq n!$. But $|D_n|=n!$, and since there is an onto homomorphism from $M_n$ onto $D_n$, this completes the proof that $M_n\cong D_n$.

\begin{remark} From \cite{Abdullahi-thesis} we know that $PD_n$ is isomorphic to $D_{n+1}$ and so there is no need to give separate presentation for $PD_n$.
\end{remark}

\section{Presentation for $ID_n$}

As with $D_n$ and $PD_n$, a general study of the semigroup $ID_n$ was initiated in \cite{Abdullahi-thesis}, see also \cite{Abdullahi2}.

Let $e_i$ be the idempotent in $ID_n$ such that $\mathrm{dom}(e_i)=\mathrm{im}(e_i)=\{1,\ldots,n\}\setminus\{i\}$; and for $1\leq i<j\leq n$ let $b_{i,j}$ be the element of $ID_n$ with $\mathrm{dom}(b_{i,j})=\{1,\ldots,n\}\setminus\{i\}$ which maps $j$ to $i$, and fixes all remaining points. One checks that the $e_i$-s together with the $b_{i,j}$-s generate $ID_n$ (see~\cite{Abdullahi2}) and that they are subject to the following relations (under the mapping $f_i\mapsto e_i$ and $a_{i,j}\mapsto b_{i,j}$):
\begin{align}
f_i^2&=f_{i} &\label{eq:a}\\
f_{i}f_{j}&=f_jf_i & i\neq j\label{eq:b}\\
f_ka_{i,j}&=a_{i,j}f_k & i<j~\&~k\notin\{i,j\}\label{eq:c}\\
f_ia_{i,j}&=a_{i,j}f_j=a_{i,j} & i<j\label{eq:d}\\
f_ja_{i,j}&=a_{i,j}f_i=f_if_j & i<j\label{eq:e}\\
a_{i,j}a_{k,l}&=a_{k,l}a_{i,j} & \{i,j\}\cap\{k,l\}=\varnothing\label{eq:f}\\
a_{j,k}a_{i,j}&=f_ja_{i,k} & i<j<k.\label{eq:g}
\end{align}
We will prove that the monoid $M_n$, presented by the relations~\eqref{eq:a}--\eqref{eq:g}, is isomorphic to $ID_n$. From~\eqref{eq:c},~\eqref{eq:d},~\eqref{eq:e} and induction, we easily see that every element $w\in M_n$ is expressible as
\begin{equation*}
f_{i_1}f_{i_2}\cdots f_{i_k}u,
\end{equation*}
where $u\in\{a_{i,j}:i<j\}^{\ast}$, $k\leq n$.
\begin{lemma}\label{lm:aa}
\begin{enumerate}[(1)]
\item
$a_{i,j}a_{k,j}=f_ka_{i,j}$ if $i,k<j$ and $i\neq k$.
\item
$a_{i,j}a_{i,k}=f_ja_{i,k}$ if $i<j,k$ and $j\neq k$.
\item
$a_{i,j}a_{i,j}=f_if_j$ if $i<j$.
\end{enumerate}
\end{lemma}

\begin{proof}
(1). $a_{i,j}a_{k,j}=a_{i,j}f_ja_{k,j}=a_{i,j}f_jf_k=a_{i,j}f_k=f_ka_{i,j}$.

(2). $a_{i,j}a_{i,k}=a_{i,j}\cdot f_ia_{i,k}=f_if_ja_{i,k}=f_ja_{i,k}$.

(3). $a_{i,j}a_{i,j}=a_{i,j}f_ia_{i,j}=f_if_ja_{i,j}=f_if_j$.
\end{proof}

From Lemma~\ref{lm:aa} it follows that any word $w\in M_n$ can be expressed as
\begin{equation}\label{eq:baba}
f_{i_1}\cdots f_{i_k}a_{t_1,j_1}\cdots a_{t_r,j_r}
\end{equation}
with $2\leq j_1<\cdots<j_r\leq n$ and $t_s<j_s$ for all $s$.

Now we will prove that additionally we may assume that in~\eqref{eq:baba} all $t_i$'s are pairwise distinct. Indeed, let~\eqref{eq:baba} have a chunk of consecutive letters $a_{t_s,j_s}\cdots a_{t_p,j_p}$ such that $t_s=t_p$ and there is no $t_i$ equal to $t_s=t_p$ for $s<i<p$. Note also that $t_s=t_p\notin\{j_s,\ldots,j_p\}$. Hence by~\eqref{eq:d},~\eqref{eq:c} and~\eqref{eq:e} we have
\begin{eqnarray*}
a_{t_s,j_s}\cdots a_{t_p,j_p} &=& a_{t_s,j_s}\cdots a_{t_{p-1},j_{p-1}}f_{t_p}a_{t_p,j_p}\\
&=& a_{t_s,j_s}\cdots a_{t_{p-2},j_{p-2}}f_{t_p}a_{t_{p-1},j_{p-1}}a_{t_p,j_p}\\
&\vdots&\\
&=& a_{t_s,j_s}f_{t_p}a_{t_{s+1},j_{s+1}}\cdots a_{t_p,j_p}\\
&=& f_{t_s}f_{j_s}a_{t_{s+1},j_{s+1}}\cdots a_{t_p,j_p},
\end{eqnarray*}
and hence we turn the product~\eqref{eq:baba} to one with lesser number of entries of $a_{i,j}$'s, which allows by use of inductive reasoning to deduce that indeed in~\eqref{eq:baba} we may assume that all $t_i$'s are pairwise distinct.

Furthermore, in~\eqref{eq:baba} we may assume that $\{i_1,\ldots,i_k\}\cap\{j_1,\ldots,j_r,t_1,\ldots,t_r\}=\varnothing$. Indeed, otherwise, using relations~\eqref{eq:b},~\eqref{eq:c} and~\eqref{eq:d} and~\eqref{eq:e}, we could push the corresponding $f_i$ to the right of the word~\eqref{eq:baba} and either replace some of $a_{t_s,j_s}$ by some $f_q$ and move that newly introduced $f_q$ back to the left, or the corresponding $f_i$ vanishes.

Now, the element~\eqref{eq:baba} in $M_n$ with all the above conditions on $i_s$, $t_l$ and $j_l$, evaluated in $ID_n$, is the element which maps $j_l$ to $t_l$ for $l\leq r$ and any point from $\{1,\ldots,n\}\setminus\{i_1,\ldots,i_{k},j_1,\ldots,j_{r},t_1,\ldots,t_{r}\}$ identically. Any such element in $ID_n$ uniquely recovers the product~\eqref{eq:baba}. Hence $|M_n|\leq |ID_n|$, but of course $|M_n|\geq |ID_n|$ and so $M_n\cong ID_n$, as required.

\section{Presentation for $C_n$}

The semigroup $C_n$ also known as the Catalan monoid because $\mid C_n\mid$ is the $n$-th Catalan number was first studied by Higgins and it also arose in language theory \cite{Higgins}. We provide for completeness the following result, the proof of which the reader can find in~\cite{SashaVolodia} and~\cite{Solomon}. The monoid $C_n$ is presented by
\begin{eqnarray}
\bigl\langle e_{i},~1\leq i\leq n-1 &:& e_i^2=e_{i}\label{eq:11}\\
&& e_{i}e_{j}=e_je_i\quad\text{if}~|i-j|\geq 2\label{eq:12}\\
&& e_ie_{i+1}e_i=e_{i+1}e_i\label{eq:13}\\
&& e_{i+1}e_{i}e_{i+1}=e_{i+1}e_i\bigr\rangle.\label{eq:14}
\end{eqnarray}

\section{Presentation for $IC_n$}

The semigroup $IC_n$ first appeared in \cite{Gan} and not much is known about it.
For $i\leq n$ let $f_i$ be the idempotent in $IC_n$ with $\mathrm{dom}(f_i)=\mathrm{im}(f_i)=\{1,\ldots,n\}\setminus\{i\}$; and for $i\leq n-1$ let $b_i$ be the element of $IC_n$ which maps $i+1$ to $i$ and fixes all the points from $\{1,\ldots,n\}\setminus\{i,i+1\}$.
\begin{lemma}
$IC_n=\langle f_i,~i\leq n;~b_i,~i\leq n-1\rangle$.
\end{lemma}

\begin{proof}
Let $f\in IC_n$. Let $F$ be the set of fixed points of $f$. Let $\mathrm{dom}(f)\setminus F$ consist of the points $j_1<\cdots<j_p$, and $\mathrm{im}(f)\setminus F$ consist of the points $t_1<\cdots<t_p$. Then one has $t_s\leq j_s$. Then one easily calculates that $$f=\prod_{x\in[n]\setminus\mathrm{dom}(f)}f_x\cdot(b_{j_1}\cdots b_{t_1})\cdots(b_{j_p}\cdots b_{t_p}).$$
\end{proof}
Also one sees that the $f_i$-s together with the $b_i$-s
satisfy the following relations (under the mapping $e_i\mapsto f_i$, $a_i\mapsto b_i$):
\begin{align}
e_i^2 &=e_{i}&\label{eq:21}\\
e_{i}e_{j}&=e_je_i&\label{eq:22}\\
e_ia_j&=a_je_i & \text{if}~i<j~\text{or}~i>j+1\label{eq:23}\\
a_ia_j&=a_ja_i & \text{if}~|i-j|\geq 2\label{eq:24}\\
e_ia_i &=a_ie_{i+1}=a_i & i\leq n-1\label{eq:25}\\
e_{i+1}a_i &=a_ie_i=e_ie_{i+1} & i\leq n-1.\label{eq:26}
\end{align}
We will prove that the monoid $M_n$, presented by the relations~\eqref{eq:21}--\eqref{eq:26}, is isomorphic to $IC_n$.
We proceed with the following
\begin{lemma}\label{lm:lalala}
\begin{enumerate}[(1)]
\item
$a_{i+1}a_ia_{i+1}=a_{i+1}a_i$.
\item
$a_ia_i=e_ie_{i+1}$.
\item
$a_ia_{i+1}a_i=a_{i+1}a_i$.
\end{enumerate}
\end{lemma}

\begin{proof}
(1).
\begin{multline*}
a_{i+1}a_ia_{i+1}=a_{i+1}e_{i+2}a_ia_{i+1}=a_{i+1}a_ie_{i+2}a_{i+1}=a_{i+1}a_ie_{i+1}e_{i+2}=\\
a_{i+1}a_ie_{i+2}=
a_{i+1}e_{i+2}a_i=a_{i+1}a_i.
\end{multline*}

(2).
\begin{equation*}
a_ia_i=a_ie_ia_i=e_ie_{i+1}a_i=e_ie_ie_{i+1}=e_ie_{i+1}.
\end{equation*}

(3).
\begin{multline*}
a_ia_{i+1}a_i=a_ia_{i+1}e_ia_i=a_ie_ia_{i+1}a_i=e_ie_{i+1}a_{i+1}a_i=\\
e_ia_{i+1}a_i=a_{i+1}e_ia_i=a_{i+1}a_i.
\end{multline*}
\end{proof}

Now, one easily sees from~\eqref{eq:23},~\eqref{eq:25} and~\eqref{eq:26} that every element $w\in M_n$ can be expressed as
\begin{equation*}
e_{i_1}\cdots e_{i_k}a_{j_1}\cdots a_{j_p}.
\end{equation*}
Actually we will prove that any $w\in M_n$ can be represented as
\begin{equation}\label{eq:hahaha}
\pi=e_{i_1}\cdots e_{i_k}(a_{j_1}\cdots a_{t_1})\cdots(a_{j_p}\cdots a_{t_p})
\end{equation}
with $j_1<j_2<\cdots<j_p$; $t_1<t_2<\cdots<t_p$; $t_s\leq j_s$.

To prove this, we proceed by induction and let $i$ be arbitrary to consider $\pi a_i$. If $i>j_p$, then $\pi a_i$ is of the required form. So let $i\leq j_p$. If $t_p\leq i$, then by Lemma~\ref{lm:lalala} and~\eqref{eq:24} we can use induction and bring $\pi a_i$ to the required form. Thus let $i<t_p$. If $i<t_p-1$, then
\begin{equation*}
\pi a_i=e_{i_1}\cdots e_{i_k}(a_{j_1}\cdots a_{t_1})\cdots(a_{j_{p-1}}\cdots a_{t_{p-1}})a_i(a_{j_{p}}\cdots a_{t_{p}}),
\end{equation*}
and we further bring $\pi a_i$ to the required form. So let $i=t_p-1$. If $t_p-1>t_{p-1}$, then $\pi a_i$ is in the needed form. Thus, let finally $t_p-1=t_{p-1}$. Then by Lemma~\ref{lm:lalala}
\begin{eqnarray*}
\pi a_i &=& e_{i_1}\cdots e_{i_k}(a_{j_1}\cdots a_{t_1})\cdots(a_{j_{p-2}}\cdots a_{t_{p-2}})(a_{j_{p-1}}\cdots a_{t_{p-1}+1})(a_{j_{p}}\cdots a_{t_{p}-1}a_{t_p}a_{t_p-1})\\
&=& e_{i_1}\cdots e_{i_k}(a_{j_1}\cdots a_{t_1})\cdots(a_{j_{p-2}}\cdots a_{t_{p-2}})(a_{j_{p-1}}\cdots a_{t_{p-1}+1})(a_{j_{p}}\cdots a_{t_p-1})\\
&\vdots&\\
&=& e_{i_1}\cdots e_{i_k}(a_{j_1}\cdots a_{t_1})\cdots(a_{j_{p-2}}\cdots a_{t_{p-2}})(a_{j_{p}}\cdots a_{t_p-1})
\end{eqnarray*}
is in the required form.

Additionally we may require that in~\eqref{eq:hahaha} the elements $i_s$ do not coincide with any of the indices $i$ of $a_i$ appearing in~\eqref{eq:hahaha}, and do not coincide with any of $j_1+1,\ldots,j_p+1$. Then all such products~\eqref{eq:hahaha} evaluated in $IC_n$ are pairwise distinct. Hence $M_n\cong IC_n$.

\section{Presentation for $PC_n$}

The semigroup $PC_n$ also known as the Schr\"{o}der monoid because $\mid PC_n\mid$ is the $n-$th (double) Schr\"{o}der number,
first appeared in \cite{Lar}) and it also arose in language theory \cite{Higgins}.

Let $a_i$ be the idempotent from $PC_n$ with $\mathrm{dom}(a_i)=\mathrm{im}(a_i)=\{1,\ldots,n\}\setminus\{i\}$; and for $i\leq n-1$ let $b_i$ be the idempotent from $PC_n$ which maps $i+1$ to $i$ and fixes all remaining points. Then the $a_i$-s together with the $b_i$-s generate $PC_n$, and are subject to the following relations (under the mapping $f_i\mapsto a_i$, $e_{i}\mapsto b_{i}$):
\begin{align}
e_i^2 &=e_{i} &\label{eq:y1}\\
e_{i}e_{j}&=e_je_i & \text{if}~|i-j|\geq 2\label{eq:y2}\\
e_ie_{i+1}e_i &=e_{i+1}e_i &\label{eq:y3}\\
e_{i+1}e_{i}e_{i+1} &=e_{i+1}e_i &\label{eq:y4}\\
f_i^2 &=f_i &\label{eq:y5}\\
f_if_j &=f_jf_i &\label{eq:y6}\\
f_je_i &=e_if_j & \mathrm{if}~j<i~\mathrm{or}~j>i+1\label{eq:y7}\\
f_{i+1}e_i &=f_{i+1} &\label{eq:y8}\\
e_if_{i+1} &=e_{i} &\label{eq:y9}\\
e_if_i &=f_if_{i+1}.&\label{eq:y10}
\end{align}
Similarly to the cases we treated above,
one shows that every element of $M_n$ can be expressed as
\begin{equation}\label{eq:Rinzai!}
\pi=f_{p_1}\cdots f_{p_r}(e_{j_1}e_{j_1-1}\cdots e_{i_1})(e_{j_2}e_{j_2-1}\cdots e_{i_2})\cdots(e_{j_k}e_{j_k-1}\cdots e_{i_k}),
\end{equation}
where $1\leq i_1<i_2<\cdots<i_k$; $j_1<j_2<\cdots<j_k$; $i_s\leq j_s$ for all $s$; $k\geq 0$; and $\{p_1,\ldots,p_r\}\cap\{j_1+1,\ldots,j_k+1\}=\varnothing$.

Then distinct words of the form~\eqref{eq:Rinzai!}, evaluated in
$PC_n$ are pairwise distinct and so $M_n\cong PC_n$.
Thus, the monoid, presented by~\eqref{eq:y1}--\eqref{eq:y10}, is isomorphic to $PC_n$.

\end{document}